\documentclass{article}
\usepackage{amssymb}
\usepackage{amsfonts}
\usepackage{amsmath}

\newtheorem{theorem}{Theorem}

\newtheorem{corollary}[theorem]{Corollary}

\newtheorem{definition}[theorem]{Definition}

\newtheorem{remark}[theorem]{Remark}

\newenvironment{proof}[1][Proof]{\noindent\textbf{#1.} }{\ \rule{0.5em}{0.5em}}
\input{tcilatex}
\begin{document}

\title{The Fekete--Szeg\"{o} Coefficient Inequality For a New\\
Class of m-Fold Symmetric Bi-Univalent Functions Satisfying Subordination
Condition}
\author{Arzu Akg\"{u}l \\
%EndAName
Department of Mathematics, Faculty of Arts and Science,\\
Kocaeli University, Kocaeli, zmit, Turkey.\\
$^{\ast }$corresponding author. e-mail: akgul@kocaeli.edu.tr}
\maketitle

\begin{abstract}
In this paper, we investigate a new subclass $S_{\Sigma _{m}}^{\varphi
,\lambda }$ of $\Sigma _{m}$ consisting of analytic and \textit{m}-fold
symmetric bi-univalent functions satisfying subordination in the open unit
disk $U$. We consider the Fekete-Szeg\"{o} inequalities for this class.
Also, we establish estimates for the coefficients for this subclas and
several related classes are also considered and connections to earlier known
results are made.


2010, Mathematics Subject Classification: 30C45, 30C50.
\end{abstract}

\section{Introduction and Definitions}

Let $A$ denote the class of functions of the form 
\begin{equation}
f(z)=z+\overset{\infty }{\underset{n=2}{\sum }}a_{n}z^{n},  \label{eq1}
\end{equation}%
which are analytic in the open unit disk $U=\left\{ z:\left\vert
z\right\vert <1\right\} $, and let $S$ be the subclass of $A$ consisting of
the form (\ref{eq1}) which are also univalent in $U.$

The Koebe one-quarter theorem \cite{Duren 83} states that the image of $U$
under every function $f$ \ from $S$ contains a disk of radius $\frac{1}{4}.$
Thus every such univalent function has an inverse $f^{-1}$ which satisfies%
\begin{equation*}
f^{-1}\left( f\left( z\right) \right) =z~~\left( z\in U\right)
\end{equation*}%
and%
\begin{equation*}
f\left( f^{-1}\left( w\right) \right) =w~~\left( \left\vert w\right\vert
<r_{0}\left( f\right) ~,~r_{0}\left( f\right) \geq \frac{1}{4}\right) ,
\end{equation*}%
where%
\begin{equation}
f^{-1}\left( w\right) =w~-a_{2}w^{2}+\left( 2a_{2}^{2}-a_{3}\right)
w^{3}-\left( 5a_{2}^{3}-5a_{2}a_{3}+a_{4}\right) w^{4}+\cdots .  \label{eq2}
\end{equation}

A function $f\in A$ is said to be bi-univalent in $U$ if both $f$ and $%
f^{-1} $ are univalent in $U.~$Let $\Sigma $ denote the class of
bi-univalent functions defined in the unit disk $U.$

For a brief history and interesting examples in the class $\Sigma ,$ see 
\cite{Srivastava 2010}. Although, the familier Koebe function is not in the
class of $\Sigma ,$there are some examples of functions member of $\Sigma $
, such as%
\begin{equation*}
\frac{z}{1-z},\ \ -\log (1-z),\ \ \ \frac{1}{2}\log \left( \frac{1+z}{1-z}%
\right)
\end{equation*}%
and so on. Other common examples of functions in $S$ for example%
\begin{equation*}
z-\frac{z^{2}}{2}\text{ and }\frac{z}{1-z^{2}}
\end{equation*}%
are also not members of $\Sigma $ (see \cite{Srivastava 2010}).

An analytic function $f$ is said to be subordinate to another analytic
function $g$, written

\begin{equation}
f(z)\prec g(z),  \label{eq3}
\end{equation}

provided that there is an analytic function $w$ defined on $U$ with 
\begin{equation*}
w(0)=0\text{ \ and \ }\left\vert w(z)\right\vert <1
\end{equation*}

satisfying the following condition:

\begin{equation*}
f(z)=g\left( w(z)\right) .
\end{equation*}

Lewin \cite{Lewin 67} is the first mathematician studied the class of
bi-univalent functions, obtaining the bound 1.51 for modulus of the second
coefficient $\left\vert a_{2}\right\vert .$ Subsequently, Brannan and Clunie 
\cite{Brannan and Clunie 80} conjectured that $\left\vert a_{2}\right\vert
\leq \sqrt{2}$ for $f\in \Sigma $ and Netanyahu \cite{Netanyahu 69} showed
that $max\left\vert a_{2}\right\vert =\frac{4}{3}.$ Brannan and Taha \cite%
{BrannanTaha} introduced certain subclasses of the bi-univalent function
class $\Sigma $ similar to the familiar subclasses. $S^{\star }\left( \beta
\right) $ and $K\left( \beta \right) $ of starlike and convex function of
order $\beta $ $\left( 0\leq \beta <1\right) $ in turn (see \cite{Netanyahu
69}). The classes $S_{\Sigma }^{\star }\left( \alpha \right) $ and $%
K_{\Sigma }\left( \alpha \right) $ of bi-starlike functions of order $\alpha 
$ and bi-convex functions of order $\alpha ,$ corresponding to the function
classes $S^{\star }\left( \alpha \right) $ and $K\left( \alpha \right) ,$
were also introduced analogously. For each of the function classes $%
S_{\Sigma }^{\star }\left( \alpha \right) $ and $K_{\Sigma }\left( \alpha
\right) ,$ they found non-sharp estimates on the initial coefficients. In
fact, the beforementioned work of Srivastava et al. \cite{Srivastava 2010}
fundamentally reviwed the investigation of diversified subclasses of the
bi-univalent function class $\Sigma $ in recent times. Recently, many
authors searched bounds for various subclasses of bi-univalent functions (%
\cite{Akgul} , \cite{Altinkaya}, \cite{Frasin 2011}, \cite{Magesh and Yamini
2013}, \cite{Srivastava 2010}, \cite{Srivastava 2013}, \cite{Xu and Gui 2012}%
). Not much is known about the bounds on the general coefficient$\
\left\vert a_{n}\right\vert $ for $n\geq 4.$ In the literature, the only a
few works determining the general coefficient bounds $\left\vert
a_{n}\right\vert $ for the analytic bi-univalent functions (\cite{Bulut 2014}%
, \cite{Hamidi and Jahangiri 2014}, \cite{Jahangiri and Hamidi 2013}). The
coefficient estimate problem for each of $\left\vert a_{n}\right\vert $ $%
\left( \ n\in 
%TCIMACRO{\U{2115} }%
%BeginExpansion
\mathbb{N}
%EndExpansion
\backslash \left\{ 1,2\right\} ;\ \ 
%TCIMACRO{\U{2115} }%
%BeginExpansion
\mathbb{N}
%EndExpansion
=\left\{ 1,2,3,...\right\} \right) $ is still an open problem.

For each function $f\in S$, the function%
\begin{equation}
h(z)=\sqrt[m]{f(z^{m})}\ \ \ \ \ \ \ \ \ \ \ (z\in U,\ \ m\in 
%TCIMACRO{\U{2115} }%
%BeginExpansion
\mathbb{N}
%EndExpansion
)  \label{eq4}
\end{equation}%
is univalent and maps the unit disk $U$ into a region with \textit{m}-fold
symmetry. A function is said to be \textit{m}-fold symmetric (see \cite%
{Koepf 89}, \cite{Pommerenke 62}) if it has the following normalized form:%
\begin{equation}
f(z)=z+\overset{\infty }{\underset{k=1}{\sum }}a_{mk+1}z^{mk+1}\ \ \ \ \ \ \
(z\in U,\ \ m\in 
%TCIMACRO{\U{2115} }%
%BeginExpansion
\mathbb{N}
%EndExpansion
).  \label{eq5}
\end{equation}

We symbolyze by $S_{m}$ the class of \textit{m}-fold symmetric univalent
functions in U, which are normalized by the series expansion (\ref{eq5}%
).Indeed, the functions in the class $S$ are \textit{one}-fold symmetric.

Similiar to the concept of \textit{m}-fold symmetric univalent functions,
here, in this work, we introduced the concept of \textit{m}-fold symmetric
bi-univalent functions. Each function $f\in \Sigma $ generates an \textit{m}%
-fold symmetric bi-univalent function for each integer $m\in 
%TCIMACRO{\U{2115} }%
%BeginExpansion
\mathbb{N}
%EndExpansion
$. The normalized form of f is given by (\ref{eq5}) and the series expansion
for $f^{-1}$ is given by below:%
\begin{eqnarray}
g\left( w\right) &=&w~-a_{m+1}w^{m+1}+\left[ \left(
m+1)a_{m+1}^{2}-a_{2m+1}\right) \right] w^{2m+1}  \label{eq6} \\
&&-\left[ \frac{1}{2}(m+1)(3m+2)a_{m+1}^{3}-(3m+2)a_{m+1}a_{2m+1}+a_{3m+1}%
\right] w^{3m+1}+\cdots .  \notag
\end{eqnarray}%
where $f^{-1}=g.$ We denote by $\Sigma _{m}$ the class of \textit{m}-fold
symmetric bi-univalent functions in $U$. Taylor Maclaurin series expansion
of the inverse function of $f^{-1}$ has been recently proven by Srivastava
et al. \cite{Srivastava 2014}. For $m=1$, the formula (\ref{eq6}) induces
the formula (\ref{eq2}) of the class $\Sigma $. Some examples of \textit{m}%
-fold symmetric bi-univalent functions are given here below:%
\begin{equation*}
\left( \frac{z^{m}}{1-z^{m}}\right) ^{\frac{1}{m}},\ \ \left[ -\log (1-z^{m})%
\right] ^{\frac{1}{m}},\ \ \ \left[ \frac{1}{2}\log \left( \frac{1+z^{m}}{%
1-z^{m}}\right) ^{\frac{1}{m}}\right] .
\end{equation*}

In this work, the class of analytic functions of the form is

\begin{equation*}
p\left( z\right) =1+p_{1}z+p_{2}z^{2}+p_{3}z^{3}+\cdots \ 
\end{equation*}

such that

\begin{equation*}
R\left( p\left( z\right) \right) >0\text{ \ \ \ \ \ }(z\in U)
\end{equation*}%
holds and this class is denoted by $\mathcal{P}$.

In the work of Pommerenke \cite{Pommerenke 62}, the m-fold symmetric
function $p$ in the class $\mathcal{P}$ is given of the form:

\begin{equation}
p\left( z\right) =1+p_{m}z+p_{2m}z^{2m}+p_{3m}z^{3m}+\cdots \   \label{eq7}
\end{equation}

\bigskip

Throughout this study, $\varphi $ will be assumed as an analytic function
with positive real part in the unit disk $U$ such that

\begin{equation*}
\varphi (0)=1\text{ \ \ \ \ \ and \ \ \ \ \ }\varphi (0)>0.
\end{equation*}

and $\varphi (U)$ is symmetric with respect to the real axis. The function $%
\varphi $ has a series expansion of the form:

\begin{equation}
\varphi \left( z\right) =1+B_{1}z+B_{2}z^{2}+B_{3}z^{3}+\cdots \ \left(
B_{1}>0\right) .  \label{eq8}
\end{equation}

Let $u(z)$ and $v(z)$ be two analytic functions in the unit disk $U$ with 
\begin{equation*}
u(0)=v(0)=0\text{ \ \ \ \ and \ \ \ \ }max\{|u(z)|,|v(z)|\}<1.
\end{equation*}

We observe that

\begin{equation}
u(z)=b_{m}z^{m}+b_{2m}z^{2m}+b_{3m}z^{3m}+\text{\textperiodcentered
\textperiodcentered \textperiodcentered }  \label{eq9}
\end{equation}

and

\begin{equation}
v(w)=c_{m}w^{m}+c_{2m}w^{2m}+c_{3m}w^{3m}+\text{\textperiodcentered
\textperiodcentered \textperiodcentered .}  \label{eq10}
\end{equation}

Also we assume that 
\begin{equation}
|b_{m}|\leq 1,\text{ \ \ \ }|b_{2m}|\leq 1-|b_{m}|^{2}\text{ \ \ \ }%
,|c_{m}|\leq 1,\text{ \ \ \ }|c_{2m}|\leq 1-|c_{m}|^{2}  \label{eq11}
\end{equation}

Making some simple calculations we can notice that

\begin{equation}
\varphi \left( u\left( z\right) \right) =1+B_{1}b_{m}z^{m}+\left(
B_{1}b_{2m}+B_{2}b_{m}^{2}\right) z^{2m}+\text{\textperiodcentered
\textperiodcentered \textperiodcentered }\left( \left\vert z\right\vert
<1\right)  \label{eq12}
\end{equation}

\bigskip and

\begin{equation}
\varphi \left( v\left( w\right) \right) =1+B_{1}c_{m}w^{m}+\left(
B_{1}c_{2m}+B_{2}c_{m}^{2}\right) z^{2m}+\text{\textperiodcentered
\textperiodcentered \textperiodcentered }\left( \left\vert w\right\vert
<1\right) .  \label{eq13}
\end{equation}

\QTP{Body Math}
\bigskip

In this study, derived substantially by the work of Ma and Minda \cite{Ma}
and \cite{Srivastava 2014}, we introduce some new subclasses of m-fold
symmetric bi-univalent functions and obtain bounds for the Taylor-Maclaurin
coefficients $\left\vert a_{m+1}\right\vert $ and $\left\vert
a_{2m+1}\right\vert $ and Fekete-Szeg\"{o} functional problems for functions
in these new classes.

\begin{definition}
\label{defn2.1}A function $f\in \Sigma _{m}$ is said to be in the class $%
S_{\Sigma _{m}}^{\lambda }(\varphi )$ if the following conditions are
satisfied:%
\begin{equation*}
f\in \Sigma _{m},
\end{equation*}%
\begin{equation}
\begin{array}{cc}
\ \ \dfrac{zf^{\prime }(z)}{(1-\lambda )f(z)+\lambda zf^{\prime }(z)}\prec
\varphi (z) & \left( 0\leq \lambda <1,\ z\in U\right)%
\end{array}%
\   \label{eq13a}
\end{equation}%
and%
\begin{equation}
\begin{array}{cc}
\dfrac{\lambda g^{\prime }(w)}{(1-\lambda )g(w)+\lambda wg^{\prime }(w)}%
\prec \varphi (w) & \left( 0\leq \lambda <1,\ w\in U\right)%
\end{array}
\label{eq13b}
\end{equation}%
where the function $g=f^{-1},$ given by the (\ref{eq6}).
\end{definition}

\bigskip

\begin{remark}
\bigskip For the case of one fold symmetric functions the class $S_{\Sigma
_{m}}^{\lambda }(\varphi )$ reduces to the following classes:
\end{remark}

\begin{enumerate}
\item In the case of $m=1$ in Definition \ref{defn2.1}, we have the class 
\begin{equation*}
S_{\Sigma _{1}}^{\lambda }(\varphi )=\mathcal{G}_{\Sigma }^{\varphi ,\varphi
}(\gamma )
\end{equation*}%
investigated by Magesh and Yamini \cite{Magesh 2014} defined by requiring
that\bigskip 
\begin{equation*}
\begin{array}{cc}
\ \ \dfrac{zf^{\prime }(z)}{(1-\lambda )f(z)+\lambda zf^{\prime }(z)}\prec
\varphi (z) & \left( 0\leq \lambda <1,\ z\in U\right)%
\end{array}%
\end{equation*}%
$\ $and%
\begin{equation*}
\begin{array}{cc}
\dfrac{\lambda g^{\prime }(w)}{(1-\lambda )g(w)+\lambda wg^{\prime }(w)}%
\prec \varphi (w) & \left( 0\leq \lambda <1,\ w\in U\right) ,%
\end{array}%
\end{equation*}%
where the function $g=f^{-1}$ given by the equation (\ref{eq2}) .

\item In the case of $m=1$ and $\lambda =0$ in Definition \ref{defn2.1},
then we have the class $S_{\Sigma _{1}}^{0}(\varphi )$ which is the class of
Ma Minda starlike functions, introduced by Ma and Minda \cite{Ma}. This
class consits of the functions 
\begin{equation*}
\dfrac{zf^{\prime }(z)}{f(z)}\prec \varphi (z)
\end{equation*}

\item[3.] In the case of $m=1$ in Definition \ref{defn2.1}, for the
different choices of the function $\varphi (z)$ , we obtain interesting
known subclasses of analytic function class.
\end{enumerate}

\bigskip

For example, If we let \ 
\begin{equation*}
\varphi (z)=\frac{1+(1-2\beta )z}{1-z}\text{ or }\varphi (z)=\frac{%
1-(1-2\beta )z}{1-z},\left( 0\leq \beta <1,\ z\in U\right) 
\end{equation*}%
then the class $S_{\Sigma _{1}}^{\lambda }(\varphi )$ reduces to the class $%
S_{\Sigma _{1}}^{0}(\varphi )=\mathcal{SS}_{\Sigma }^{\ast }(\beta ,\lambda )
$. This class contains the functions satisfying the conditions\bigskip 
\begin{equation*}
\begin{array}{cc}
\ \ \left\vert \arg \left( \dfrac{zf^{\prime }(z)}{(1-\lambda )f(z)+\lambda
zf^{\prime }(z)}\right) \right\vert <\dfrac{\alpha \pi }{2} & \left(
0<\alpha \leq 1,0\leq \lambda <1,\ z\in U\right) 
\end{array}%
\ 
\end{equation*}%
and%
\begin{equation*}
\begin{array}{cc}
\left\vert \arg \left( \dfrac{\lambda g^{\prime }(w)}{(1-\lambda
)g(w)+\lambda wg^{\prime }(w)}\right) \right\vert <\dfrac{\alpha \pi }{2} & 
\left( 0<\alpha \leq 1,0\leq \lambda <1,\ w\in U\right) .%
\end{array}%
\end{equation*}

Similarly, \i f we let%
\begin{equation*}
\varphi (z)=\left( \frac{1+z}{1-z}\right) ^{\alpha }\text{ or }\varphi
(z)=\left( \frac{1-z}{1+z}\right) ^{\alpha },(0<\alpha \leq 1,z\in U)
\end{equation*}%
then the class $S_{\Sigma _{1}}^{\lambda }(\varphi )$ reduces to the class $%
S_{\Sigma _{1}}^{0}(\varphi )=\mathcal{SS}_{\Sigma }^{\ast }(\alpha ,\lambda
)$ . This class contains the functions satisfying the conditions 
\begin{equation*}
\begin{array}{cc}
f\in \Sigma ,\ \ \func{Re}\left( \dfrac{zf^{\prime }(z)}{(1-\lambda
)f(z)+\lambda zf^{\prime }(z)}\right) >\beta  & \left( 0\leq \beta <1,0\leq
\lambda <1,\ z\in U\right) 
\end{array}%
\ 
\end{equation*}%
and%
\begin{equation*}
\begin{array}{cc}
\func{Re}\left( \dfrac{wg^{\prime }(w)}{(1-\lambda )g(w)+\lambda wg^{\prime
}(w)}\right) >\beta  & \left( 0\leq \beta <1,0\leq \lambda <1,\ w\in
U\right) .%
\end{array}%
\end{equation*}
The classes $\mathcal{G}_{\Sigma }^{{}}(\beta ,\lambda )$ and $\mathcal{G}%
_{\Sigma }^{{}}(\alpha ,\lambda )$ \ were introduced and studied by
Murugusundaramoorthy et al. ( see Definition 1 and Definition 2 in \cite%
{Murugusundaramoorthy 2013}). \ Also if we choose $\lambda =0,$ $\mathcal{G}%
_{\Sigma }^{{}}(\beta ,0):$ $\mathcal{G}_{\Sigma }^{{}}\beta )$ and $%
\mathcal{G}_{\Sigma }^{{}}(\alpha ,0):\mathcal{G}_{\Sigma }^{{}}(\alpha ).$%
These classes are called bi-starlike functions of order $\beta $ and
strongly bi-starlike functions of order $\alpha $\ , respectively. The
classes $\mathcal{G}_{\Sigma }^{{}}(\beta )$ and $\mathcal{G}_{\Sigma
}^{{}}(\alpha )$ were investigated and studied by Brannan and Taha \cite%
{Brannan Taha} ( see Definition 1.1 and Definition 1.2) .

\bigskip

If we set $\lambda =0$ in Definition1, then the class $S_{\Sigma
_{m}}^{\lambda }(\varphi )$ reduces to the class $S_{\Sigma _{m}}^{\varphi }$
defined by below:

\begin{definition}
\label{defn2.2}\bigskip A function $f\in \Sigma _{m}$ is said to be in the
class $S_{\Sigma _{m}}^{\varphi }$ if the following conditions are satisfied:%
\begin{equation*}
f\in \Sigma _{m},
\end{equation*}%
\begin{equation*}
\begin{array}{cc}
\ \ \dfrac{zf^{\prime }(z)}{f(z)}\prec \varphi (z) & \left( 0\leq \lambda
<1,\ z\in U\right)%
\end{array}%
\ 
\end{equation*}%
and%
\begin{equation*}
\begin{array}{cc}
\dfrac{wg^{\prime }(w)}{g(w)}\prec \varphi (w) & \left( 0\leq \lambda <1,\
w\in U\right)%
\end{array}%
\end{equation*}%
where the function $g=f^{-1}$given by (\ref{eq6}).
\end{definition}

\bigskip

\begin{remark}
In the case of $m=1$ in Definition \ref{defn2.2}, it is interesting that,
also for $\lambda =0$ , the classes $S_{\Sigma _{1}}^{\lambda }(\frac{%
1+(1-2\beta )z}{1-z})$ and $S_{\Sigma _{1}}^{\lambda }(\left( \frac{1+z}{1-z}%
\right) ^{\alpha })$ \ lead the class $\delta _{\Sigma }^{\alpha }(\beta )$\
of bi-starlike functions of order $\alpha $ and $\delta _{\Sigma }^{\ast
}(\beta )$\ of bi-starlike functions of order $\beta $ , respectively.
\end{remark}

\bigskip 

For $m-fold$ \ symmetric ananlytic and bi-univalent functions, Alt\i nkaya
and Yal\c{c}\i n \cite{Altinkaya} \ defined and investigated the function
classes $S_{\Sigma _{m}}^{{}}(\beta ,\lambda )$ and $S_{\Sigma
_{m}}^{{}}(\alpha ,\lambda )$ as following.\bigskip 

\ A function $f\in \Sigma _{m}$ is said to be in the class $S_{\Sigma
_{m}}(\alpha ,\lambda )$ if the following conditions are satisfied:%
\begin{equation*}
\begin{array}{cc}
\ \ \left\vert \arg \left( \dfrac{zf^{\prime }(z)}{(1-\lambda )f(z)+\lambda
zf^{\prime }(z)}\right) \right\vert <\dfrac{\alpha \pi }{2} & \left(
0<\alpha \leq 1,0\leq \lambda <1,\ z\in U\right) 
\end{array}%
\ 
\end{equation*}%
and%
\begin{equation*}
\begin{array}{cc}
\left\vert \arg \left( \dfrac{\lambda g^{\prime }(w)}{(1-\lambda
)g(w)+\lambda wg^{\prime }(w)}\right) \right\vert <\dfrac{\alpha \pi }{2} & 
\left( 0<\alpha \leq 1,0\leq \lambda <1,\ w\in U\right) .%
\end{array}%
\end{equation*}

\bigskip In the same way, the function $f\in \Sigma _{m}$ is said to be in
the class $S_{\Sigma _{m}}(\beta ,\lambda )$ if the following conditions are
satisfied:%
\begin{equation*}
\begin{array}{cc}
f\in \Sigma ,\ \ \func{Re}\left( \dfrac{zf^{\prime }(z)}{(1-\lambda
)f(z)+\lambda zf^{\prime }(z)}\right) >\beta & \left( 0\leq \beta <1,0\leq
\lambda <1,\ z\in U\right)%
\end{array}%
\ 
\end{equation*}%
and%
\begin{equation*}
\begin{array}{cc}
\func{Re}\left( \dfrac{wg^{\prime }(w)}{(1-\lambda )g(w)+\lambda wg^{\prime
}(w)}\right) >\beta & \left( 0\leq \beta <1,0\leq \lambda <1,\ w\in U\right)%
\end{array}%
\end{equation*}%
where the function $g=f^{-1}$ given by (\ref{eq2}).

\begin{theorem}
\cite{Altinkaya} Let $\ f$ given by (\ref{eq4}) be in the class $S_{\Sigma
_{m}}(\alpha ,\lambda ),\ 0<\alpha \leq 1.$ Then 
\begin{equation*}
\left\vert a_{m+1}\right\vert \leq \frac{2\alpha }{m(1-\lambda )\sqrt{\alpha
+1}}
\end{equation*}%
and%
\begin{equation*}
\left\vert a_{2m+1}\right\vert \leq \frac{\alpha }{m(1-\lambda )}+\frac{%
2(m+1)\alpha ^{2}}{m^{2}(1-\lambda )^{2}}.
\end{equation*}
\end{theorem}

\begin{theorem}
\cite{Altinkaya} Let $\ f$ given by (\ref{eq4}) be in the class $S_{\Sigma
_{m}}(\beta ,\lambda ),0\leq \beta <1.$ Then 
\begin{equation*}
\left\vert a_{m+1}\right\vert \leq \frac{\sqrt{2(1-\beta )}}{m(1-\lambda )}
\end{equation*}%
and%
\begin{equation*}
\left\vert a_{2m+1}\right\vert \leq \frac{(1-\beta )}{m(1-\lambda )}+\frac{%
2(m+1)\alpha ^{2}}{m^{2}(1-\lambda )^{2}}.
\end{equation*}
\end{theorem}

\section{\protect\bigskip The Main Results and Their Consequences}

\begin{theorem}
\label{thm2.1}Let $\ f$ given by (\ref{eq5}) be in the class $S_{\Sigma
_{m}}^{\lambda }(\varphi ).$ Then 
\begin{equation}
\left\vert a_{m+1}\right\vert \leq \frac{B_{1}\sqrt{B_{1}}}{m(1-\lambda )%
\sqrt{\left\vert B_{1}^{2}-2B_{2}\right\vert +B_{1}}}  \label{eq14}
\end{equation}%
and%
\begin{equation}
\left\vert a_{2m+1}\right\vert \leq \left\{ 
\begin{array}{c}
\left( m+1-\frac{m(1-\lambda )}{B_{1}}\right) \frac{B_{1}^{3}}{%
2m^{2}(1-\lambda )^{2}\left[ B_{1}+\left\vert B_{1}^{2}-2B_{2}\right\vert %
\right] }+\frac{B_{1}}{2m(1-\lambda )},\text{ \ \ \ }B_{1}\geq \frac{%
m(1-\lambda )}{m+1} \\ 
\text{ \ \ \ }\frac{B_{1}}{2m(1-\lambda )},\text{ \ \ \ \ \ \ \ \ \ \ \ \ \
\ \ \ \ \ \ \ \ \ \ \ \ \ \ \ \ \ \ \ \ \ \ \ \ \ \ \ \ \ \ \ \ \ \ \ \ \ \
\ \ \ \ \ \ \ \ \ \ \ \ \ }B_{1}<\frac{m(1-\lambda )}{m+1}%
\end{array}%
\right.  \label{eq15}
\end{equation}
\end{theorem}

\begin{proof}
Let $\ f\in S_{\Sigma _{m}}^{\lambda }(\varphi ).$ Then there are analytic
functions $u:U\rightarrow U$ and $v:U\rightarrow U$, with%
\begin{equation*}
u(0)=v(0)=0,
\end{equation*}%
satisfying the following conditions:

\begin{equation}
\begin{array}{c}
\dfrac{zf^{\prime }(z)}{(1-\lambda )f(z)+\lambda zf^{\prime }(z)}=\varphi
(u(z)) \\ 
\text{and} \\ 
\dfrac{\lambda g^{\prime }(w)}{(1-\lambda )g(w)+\lambda wg^{\prime }(w)}%
=\varphi (v(w)).%
\end{array}
\label{eq16}
\end{equation}

Using the equalities (\ref{eq12}), (\ref{eq13}) in (\ref{eq16}) and
comparing the coefficient of (\ref{eq16}) , we have%
\begin{equation}
m(1-\lambda )a_{m+1}=B_{1}b_{m},  \label{eq17}
\end{equation}%
\begin{equation}
m(1-\lambda )\left[ 2a_{2m+1}-(\lambda m+1)a_{m+1}^{2}\right]
=B_{1}b_{2m}+B_{2}b_{m}^{2},  \label{eq18}
\end{equation}%
and%
\begin{equation}
-m(1-\lambda )a_{m+1}=B_{1}c_{m},  \label{eq19}
\end{equation}%
\begin{equation}
m(1-\lambda )\left[ \left( 1+m(2-\lambda ))a_{m+1}^{2}-2a_{2m+1}\right) %
\right] =B_{1}c_{2m}+B_{2}c_{m}^{2}.  \label{eq20}
\end{equation}%
From (\ref{eq17}) and (\ref{eq19}) we find that 
\begin{equation}
b_{m}=-c_{m}.  \label{eq21}
\end{equation}%
Adding (\ref{eq18}) and (\ref{eq20}), we get 
\begin{equation}
2m^{2}(1-\lambda )^{2}a_{m+1}^{2}=B_{1}\left( b_{2m}+c_{2m}\right)
+B_{2}(b_{m}^{2}+c_{m}^{2}).  \label{eq22}
\end{equation}%
and using the relation (\ref{eq17}) and (\ref{eq21}) in (\ref{eq22}), we
have 
\begin{equation*}
2m^{2}(1-\lambda )^{2}a_{m+1}^{2}=B_{1}\left( b_{2m}+c_{2m}\right) +\frac{%
2B_{2}m^{2}(1-\lambda )^{2}}{B_{1}^{2}}a_{m+1}^{2}.
\end{equation*}%
Therefore, by a simple calculation we get 
\begin{equation}
2m^{2}(1-\lambda )^{2}\left( B_{1}^{2}-2B_{2}\right)
a_{m+1}^{2}=B_{1}^{3}\left( b_{2m}+c_{2m}\right)  \label{eq23}
\end{equation}%
By using the inequalities given by (\ref{eq11}) in (\ref{eq23}) for the
coefficients $b_{2m}$ and $c_{2m}$, we obtain%
\begin{equation}
\left\vert 2m^{2}(1-\lambda )^{2}\left( B_{1}^{2}-2B_{2}\right)
a_{m+1}^{2}\right\vert \leq 2B_{1}^{3}\left( 1-\left\vert
b_{m}^{2}\right\vert \right) .  \label{eq24}
\end{equation}%
Also by using (\ref{eq17}) in (\ref{eq24}) we have

\begin{equation*}
\left\vert a_{m+1}\right\vert ^{2}\leq \frac{B_{1}^{3}}{m^{2}(1-\lambda
)^{2}\left( \left\vert B_{1}^{2}-2B_{2}\right\vert +B_{1}\right) }
\end{equation*}%
which implies the assertion (\ref{eq14}).

Next, in order to find the bound on $\left\vert a_{2m+1}\right\vert ,$ by
subtracting (\ref{eq20}) from (\ref{eq18}), we obtain%
\begin{equation}
4m(1-\lambda )a_{2m+1}=2m(m+1)(1-\lambda )a_{m+1}^{2}+B_{1}\left(
b_{2m}-c_{2m}\right) .  \label{eq25}
\end{equation}%
Then, in view of (\ref{eq17}) ,(\ref{eq21}) and (\ref{eq25}) , applying the
inequalities in (\ref{eq11}) for the coefficients $p_{2m},p_{m}$ $q_{m\text{
\ }}$and $q_{2m},$we have 
\begin{equation*}
\left\vert a_{2m+1}\right\vert \leq \left( (m+1)-\frac{m(1-\lambda )}{B_{1}}%
\right) \frac{B_{1}^{3}}{2m^{2}(1-\lambda )^{2}\left[ B_{1}+\left\vert
B_{1}^{2}-2B_{2}\right\vert \right] }+\frac{B_{1}}{2m(1-\lambda )}
\end{equation*}%
which implies the assertion (\ref{eq15}) .This completes the proof of
Theorem \ref{thm2.1}.
\end{proof}

For $\lambda =0,$ we can state the following corollary:

\begin{corollary}
Let $\ f$ given by (\ref{eq5}) be in the class $S_{\Sigma _{m}}^{{}}(\varphi
).$ Then 
\begin{equation}
\left\vert a_{m+1}\right\vert \leq \frac{B_{1}\sqrt{B_{1}}}{m\sqrt{%
\left\vert B_{1}^{2}-2B_{2}\right\vert +B_{1}}}
\end{equation}%
and%
\begin{equation}
\left\vert a_{2m+1}\right\vert \leq \left\{ 
\begin{array}{c}
\left( m+1-\frac{m}{B_{1}}\right) \frac{B_{1}^{3}}{2m^{2}\left[
B_{1}+\left\vert B_{1}^{2}-2B_{2}\right\vert \right] }+\frac{B_{1}}{2m)},%
\text{ \ \ \ }B_{1}\geq \frac{m}{m+1} \\ 
\text{ \ \ \ }\frac{B_{1}}{2m},\text{ \ \ \ \ \ \ \ \ \ \ \ \ \ \ \ \ \ \ \
\ \ \ \ \ \ \ \ \ \ \ \ \ \ \ \ \ \ \ \ \ \ \ \ \ \ \ \ \ \ \ \ \ \ \ \ \ \
\ \ \ \ \ \ \ }B_{1}<\frac{m}{m+1}%
\end{array}%
\right.
\end{equation}
\end{corollary}

\bigskip

For one fold symmetric functions, we obtain the following corollaries:

\begin{corollary}
If the function $f\in \Sigma $ is in the class of $S_{\Sigma _{1}}^{\lambda
}(\varphi )=\mathcal{G}_{\Sigma }^{\varphi ,\varphi }(\gamma ),$ then 
\begin{equation}
\left\vert a_{2}\right\vert \leq \frac{B_{1}\sqrt{B_{1}}}{(1-\lambda )\sqrt{%
\left\vert B_{1}^{2}-2B_{2}\right\vert +B_{1}}}
\end{equation}%
and%
\begin{equation}
\left\vert a_{3}\right\vert \leq \left\{ 
\begin{array}{ccc}
\left( 2-\frac{(1-\lambda )}{B_{1}}\right) \frac{B_{1}^{3}}{2(1-\lambda
)^{2}+B_{1}\left\vert B_{1}^{2}-2B_{2}\right\vert }+\frac{B_{1}}{2(1-\lambda
)},\text{ } &  & \text{for \ \ \ \ \ \ \ }B_{1}\geq \frac{1-\lambda }{2} \\ 
\text{\ \ \ }\frac{B_{1}}{2(1-\lambda )},\text{ \ \ \ \ } &  & \text{for \ \
\ \ \ \ \ }B_{1}<\frac{1-\lambda }{2}%
\end{array}%
\right. .
\end{equation}
\end{corollary}

\bigskip 

For the case of one-fold symmetric functions and for $\varphi (z)=\left( 
\frac{1+z}{1-z}\right) ^{\alpha }=1+2\alpha z+2\alpha ^{2}z^{2}+...,$
Theorem 1 reduces to the following result:

\begin{corollary}
If the function $f\in \Sigma $ is in the class of $S_{\Sigma _{1}}^{\lambda
}(\left( \frac{1+z}{1-z}\right) ^{\alpha }),$ then%
\begin{equation}
\left\vert a_{2}\right\vert \leq \frac{2\alpha }{(1-\lambda )}
\end{equation}%
and%
\begin{equation}
\left\vert a_{3}\right\vert \leq \left\{ 
\begin{array}{ccc}
\frac{4\alpha ^{2}}{(1-\lambda )^{2}},\text{ } &  & \text{for \ \ \ \ \ \ \ }%
\alpha \geq \frac{1-\lambda }{4} \\ 
\text{\ \ \ }\frac{\alpha }{(1-\lambda )},\text{ \ \ \ \ } &  & \text{for \
\ \ \ \ \ \ }\alpha <\frac{1-\lambda }{4}%
\end{array}%
\right. .
\end{equation}
\end{corollary}

The estimates for $\left\vert a_{2}\right\vert $ and $\left\vert
a_{3}\right\vert $ asserted by Corollary? more accurate than those given by
Corollary 1 in Magesh and Yamini.

\bigskip 

For the case  of one-fold symmetric functions and for $\varphi (z)=\frac{%
1+(1-2\beta )z}{1-z}=1+2(1-\beta )z+(1-\beta )z^{2}+...,$Theorem 1 reduces
to the following result:

\begin{corollary}
If the function $f\in \Sigma $ is in the class of $S_{\Sigma _{1}}^{\lambda
}(\frac{1+(1-2\beta )z}{1-z}),$ then 
\begin{equation}
\left\vert a_{2}\right\vert \leq \frac{2(1-\beta )}{(1-\lambda )\sqrt{2\beta
+1}}
\end{equation}%
and%
\begin{equation}
\left\vert a_{3}\right\vert \leq \left\{ 
\begin{array}{ccc}
\frac{4\alpha ^{2}}{(1-\lambda )^{2}},\text{ } &  & \text{for \ \ \ \ \ \ \ }%
\alpha \geq \frac{1-\lambda }{4} \\ 
\text{\ \ \ }\frac{\alpha }{(1-\lambda )},\text{ \ \ \ \ } &  & \text{for \
\ \ \ \ \ \ }\alpha <\frac{1-\lambda }{4}%
\end{array}%
\right. .
\end{equation}
\end{corollary}

The estimates for $\left\vert a_{2}\right\vert $ and $\left\vert
a_{3}\right\vert $ asserted by Corollary? more accurate than those given by
Corollary 2 in Magesh and Yamini.

\bigskip 

Also, if we choose $\lambda =0$ in Corollary 7, we have the following
corollary:

\begin{corollary}
If the function $f\in \Sigma $ is in the class of $S_{\Sigma
_{1}}^{0}(\varphi )$ then 
\begin{equation}
\left\vert a_{2}\right\vert \leq \frac{B_{1}\sqrt{B_{1}}}{\sqrt{\left\vert
B_{1}^{2}-2B_{2}\right\vert +B_{1}}}
\end{equation}
\end{corollary}

\begin{equation}
\left\vert a_{3}\right\vert \leq \left\{ 
\begin{array}{ccc}
\left( 2-\frac{1}{B_{1}}\right) \frac{B_{1}^{3}}{2+B_{1}\left\vert
B_{1}^{2}-2B_{2}\right\vert }+\frac{B_{1}}{2},\text{ } &  & \text{for \ \ \
\ \ \ \ }B_{1}\geq \frac{1}{2} \\ 
\text{\ \ \ }\frac{B_{1}}{2},\text{ \ \ \ \ } &  & \text{for \ \ \ \ \ \ \ }%
B_{1}<\frac{1}{2}%
\end{array}%
\right. .
\end{equation}

For one- fold symmetric functions, if we choose the function $\varphi (z)$
in different forms, then we have the following corollaries. named Corollary
13 and Corollary 14:

\begin{corollary}
\cite{Murugusundaramoorthy 2013} Let the function $f(z)$ given by the
equality (\ref{eq1}) be in the class $\mathcal{SS}_{\Sigma }^{\ast }(\beta
,\lambda ),0\leq \beta <1$ and $0\leq \lambda <1.$ Then 
\begin{equation*}
\left\vert a_{2}\right\vert \leq \frac{2\sqrt{(1-\beta )}}{(1-\lambda )}
\end{equation*}%
and 
\begin{equation*}
\left\vert a_{3}\right\vert \leq \frac{4(1-\beta )^{2}}{(1-\lambda )^{2}}+%
\frac{(1-\beta )}{(1-\lambda )}.
\end{equation*}
\end{corollary}

\bigskip

\begin{corollary}
\cite{Murugusundaramoorthy 2013} Let the function $f(z)$ given by the
equality (\ref{eq1}) be in the class $\mathcal{SS}_{\Sigma }^{\ast }(\alpha
,\lambda ),0<\alpha \leq 1$ and $0\leq \lambda <1.$ Then
\end{corollary}

\begin{equation*}
\left\vert a_{2}\right\vert \leq \frac{2\alpha }{(1-\lambda )\sqrt{1+\alpha }%
}
\end{equation*}%
and 
\begin{equation*}
\left\vert a_{3}\right\vert \leq \frac{4\alpha ^{2}}{(1-\lambda )^{2}}+\frac{%
\alpha }{(1-\lambda )}.
\end{equation*}

\bigskip \qquad \qquad \qquad \qquad \qquad

\bigskip For \textit{one}-fold symmetric bi-univalent functions and $\lambda
=0$, Theorem 5 reduces to Corollary which were proven earlier by
Murugunsundaramoorthy et al. \cite{Murugusundaramoorthy 2013}

\bigskip

\begin{corollary}
Let $\ f$ given by (\ref{eq4}) be in the class $S_{\Sigma }^{\ast }(\alpha
)\ \ \left( 0<\alpha \leq 1\right) $. Then 
\begin{equation*}
\left\vert a_{2}\right\vert \leq \frac{2\alpha }{\sqrt{\alpha +1}}
\end{equation*}%
and%
\begin{equation*}
\left\vert a_{3}\right\vert \leq 4\alpha ^{2}+\alpha .
\end{equation*}
\end{corollary}

\bigskip

Here, in this study, we will spesify the theorem concerning the Fekete- Szeg%
\"{o} inequality for the class $S_{\Sigma _{m}}^{\lambda }(\varphi ).$To
improve the result, especially Theorem 2.1, we consider Fekete-Szeg\"{o}
inequality for the class $S_{\Sigma _{m}}^{\lambda }(\varphi )$ . This kind
of studies has been made by many authors. The results regarding this problem
are given in the works of \cite{Choi}, \cite{Kanas}, \cite{London}, \cite%
{Srivastava2001}. The conclutions given in the study are not sharp, but,
unfortunatelly, there isn't any method giving sharp results as regards these
problems.

\begin{theorem}
\label{thm2.2}Let $\ f$ given by (\ref{eq4}) be in the class $S_{\Sigma
_{m}}^{\lambda }(\varphi ).$ Then
\end{theorem}

\begin{equation}
\left\vert a_{2m+1}-\gamma a_{m+1}^{2}\right\vert \leq \left\{ 
\begin{array}{ccc}
\begin{array}{c}
\frac{B_{1}}{2m(1-\lambda )} \\ 
\end{array}
&  & 
\begin{array}{c}
\text{for \ \ \ \ \ \ \ }0\leq \left\vert h(\gamma )\right\vert <\frac{1}{%
4m(1-\lambda )} \\ 
\end{array}
\\ 
2B_{1}\left\vert h(\gamma )\right\vert &  & \text{for \ \ \ \ \ \ \ }%
\left\vert h(\gamma )\right\vert \geq \frac{1}{4m(1-\lambda )}%
\end{array}%
\right.  \label{eq26}
\end{equation}

\bigskip 
\begin{equation*}
h(\gamma )=\left( \frac{m+1-2\gamma }{2}\right) \frac{B_{1}^{2}}{%
2m^{2}(1-\lambda )^{2}\left( B_{1}^{2}-2B_{2}\right) }
\end{equation*}

\begin{proof}
\bigskip From the equations (\ref{eq23}) and (\ref{eq25}) ,%
\begin{equation}
a_{m+1}^{2}=\frac{B_{1}^{3}\left( b_{2m}+c_{2m}\right) }{2m^{2}(1-\lambda
)^{2}\left( B_{1}^{2}-2B_{2}\right) }  \label{eq27}
\end{equation}%
and
\end{proof}

\begin{equation}
a_{2m+1}=\frac{m+1}{2}a_{m+1}^{2}-\frac{B_{1}\left( b_{2m}-c_{2m}\right) }{%
4m(1-\lambda )}  \label{eq28}
\end{equation}

\bigskip By using the equalities (\ref{eq27}) and (\ref{eq28}), we have 
\begin{equation*}
a_{2m+1}-\gamma a_{m+1}^{2}=B_{1}\left[ \left( h(\gamma )+\frac{1}{%
4m(1-\lambda )}\right) b_{2m}+\left( h(\gamma )-\frac{1}{4m(1-\lambda )}%
\right) c_{2m}\right]
\end{equation*}%
where

\begin{equation*}
h(\gamma )=\left( \frac{m+1-2\gamma }{2}\right) \frac{B_{1}^{2}}{%
2m^{2}(1-\lambda )^{2}\left( B_{1}^{2}-2B_{2}\right) }
\end{equation*}

\bigskip Due to the fact that all $B_{i}$ are real and $B_{1}>0$, which
holds the assertion (\ref{eq26}), the proof of the theorem is copmleted.

For m-fold symmetric functions, if we choose $\lambda =0$ in the Theorem \ref%
{thm2.2}, we obtain the following corollary:

\begin{corollary}
Let $\ f$ given by (\ref{eq4}) be in the class $S_{\Sigma _{m}}^{{}}(\varphi
).$ Then
\end{corollary}

\begin{equation}
\left\vert a_{2m+1}-\gamma a_{m+1}^{2}\right\vert \leq \left\{ 
\begin{array}{ccc}
\begin{array}{c}
\frac{B_{1}}{2m} \\ 
\end{array}
&  & 
\begin{array}{c}
\text{for \ \ \ \ \ \ \ }0\leq \left\vert h(\gamma )\right\vert <\frac{1}{4m}
\\ 
\end{array}
\\ 
2B_{1}\left\vert h(\gamma )\right\vert &  & \text{for \ \ \ \ \ \ \ }%
\left\vert h(\gamma )\right\vert \geq \frac{1}{4m}%
\end{array}%
\right.
\end{equation}

\bigskip where%
\begin{equation*}
h(\gamma )=\left( \frac{m+1-2\gamma }{2}\right) \frac{B_{1}^{2}}{%
2m^{2}\left( B_{1}^{2}-2B_{2}\right) }
\end{equation*}%
For one-fold symmetric functions, we can state the Fekete-Szeg\"{o}
inequality for the class $S_{\Sigma _{1}}^{\lambda }(\varphi )$ as following:

\begin{corollary}
If the function $f\in \Sigma $ is in the class of $S_{\Sigma _{1}}^{\lambda
}(\varphi )=\mathcal{G}_{\Sigma _{m}}^{\varphi ,\varphi }(\gamma )$ , then
we get%
\begin{equation*}
\left\vert a_{3}-\gamma a_{2}^{2}\right\vert \leq \left\{ 
\begin{array}{ccc}
\frac{B_{1}}{4(1-\lambda )} &  & \text{for \ \ \ \ \ \ \ }0\leq \left\vert
h(\gamma )\right\vert <\frac{1}{4(1-\lambda )} \\ 
4B_{1}\left\vert h(\gamma )\right\vert &  & \text{for \ \ \ \ \ \ \ }%
\left\vert h(\gamma )\right\vert \geq \frac{1}{4(1-\lambda )}%
\end{array}%
\right. .
\end{equation*}%
where%
\begin{equation*}
h(\gamma )=\left( 1-\gamma \right) \frac{B_{1}^{2}}{2(1-\lambda )^{2}\left(
B_{1}^{2}-2B_{2}\right) }
\end{equation*}
\end{corollary}

\bigskip If we choose $\lambda =0$ in Corollary 15, then we have the
following corollary:

\begin{corollary}
If the function $f\in \Sigma $ is in the class of $S_{\Sigma
_{1}}^{0}(\varphi )$ then%
\begin{equation*}
\left\vert a_{3}-\gamma a_{2}^{2}\right\vert \leq \left\{ 
\begin{array}{ccc}
\frac{B_{1}}{4} &  & \text{for \ \ \ \ \ \ \ }0\leq \left\vert h(\gamma
)\right\vert <\frac{1}{4} \\ 
4B_{1}\left\vert h(\gamma )\right\vert &  & \text{for \ \ \ \ \ \ \ }%
\left\vert h(\gamma )\right\vert \geq \frac{1}{4}%
\end{array}%
\right. .
\end{equation*}%
where%
\begin{equation*}
h(\gamma )=\left( 1-\gamma \right) \frac{B_{1}^{2}}{2\left(
B_{1}^{2}-2B_{2}\right) }.
\end{equation*}
\end{corollary}

\bigskip

Choosing $\gamma =1$ and $\gamma =0$ in Theorem 13, we obtain following
corollary:

\bigskip

\begin{corollary}
Let $\ f$ given by (\ref{eq5}) be in the class $S_{\Sigma _{m}}^{\lambda
}(\varphi ).$ Then
\end{corollary}

\begin{equation*}
\left\vert a_{2m+1}-\gamma a_{m+1}^{2}\right\vert \leq \left\{ 
\begin{array}{ccc}
\frac{B_{1}}{4m(1-\lambda )} &  & \text{for \ \ \ \ \ \ \ }0\leq \left\vert
h(\gamma )\right\vert <\frac{1}{4m(1-\lambda )} \\ 
4B_{1}\left\vert h(\gamma )\right\vert &  & \text{for \ \ \ \ \ \ \ }%
\left\vert h(\gamma )\right\vert \geq \frac{1}{4m(1-\lambda )}%
\end{array}%
\right. .
\end{equation*}

\bigskip

Choosing $\gamma =1$ and $m=1$ in Corollary17, we have the following
corollary:

\begin{corollary}
Let $\ f$ given by (\ref{eq5}) be in the class $S_{\Sigma _{,1}}^{\varphi
,\lambda }.$ Then
\end{corollary}

\begin{equation*}
\left\vert a_{3}-a_{2}^{2}\right\vert \leq \frac{B_{1}}{4(1-\lambda )}.
\end{equation*}%
Also, if we choose \bigskip $\lambda =0,$ then we have 
\begin{equation*}
\left\vert a_{3}-a_{2}^{2}\right\vert \leq \frac{B_{1}}{4}.
\end{equation*}

\textbf{Conclusion}

In this study, we have composed of several new subclasses of m-fold
symmetric bi-univalent analytic functions by means of subordination. For
functions belonging to the clsssess introduced here, we have obtained
inequalities on the Taylor Maclaurin coefficients $\left\vert
a_{m+1}\right\vert $ and $\left\vert a_{2m+1}\right\vert .$ Also, for
functions contained these classes , we find Fekete-Szeg\"{o} inequalities
and we have made some connections to some of earlier known results .

\bigskip 

\textbf{Author's contributions}

All author worked on the results and read and approved the final manuscript.

\textbf{Competing interests}

The authors declare that they have no competing interests.

\textbf{Acknowledgement\bigskip }

The author would like to thank the anonymous referees for their valuable
suggestions which improved

the presentation of the paper.

\end{document}